\documentclass[10pt,a4paper]{amsart}
\usepackage{a4wide,amsmath,amsfonts,amssymb,amsthm,mathrsfs,setspace,bm,amsxtra,mathtools,wasysym,textcomp} 
\mathtoolsset{showonlyrefs}

\usepackage[english]{babel}
\usepackage[all]{xy}
 \usepackage{verbatim}
\usepackage{todonotes}

\newtheorem{thm}{Theorem}[section]
\newtheorem{prop}[thm]{Proposition}
\newtheorem{lemma}[thm]{Lemma}
\newtheorem{defin}[thm]{Definition}
\newtheorem{cor}[thm]{Corollary}

\numberwithin{equation}{section}
\theoremstyle{definition}
\newtheorem{conjecture}{Conjecture}

\newtheorem{exa}[thm]{Example} 
\newenvironment{ex}{\begin{exa}}{\parbox{2mm}{\hfill}\hfill $\triangle$\end{exa}}
\newtheorem{rema}[thm]{Remark} 
\newenvironment{rem}{\begin{rema}}{\parbox{2mm}{\hfill}\hfill $\triangle$\end{rema}}

\newcommand{\rk}{\operatorname{rk}}

\newcommand{\End}{\operatorname{End}}

\newcommand{\im}{\operatorname{Im}}
\newcommand{\Pic}{\operatorname{Pic}}
\newcommand{\Ol}{\mathcal{O}}

\newcommand{\id}{\operatorname{id}}

\newcommand{\Z}{\mathbb{Z}}
\newcommand{\Q}{\mathbb{Q}}

\newcommand{\Com}{\mathbb{C}}

\newcommand{\Ea}{\mathfrak{E}}

\newcommand{\gr}{\operatorname{gr}}
\newcommand{\om}[1]{\Omega_{#1}^1}

\begin{document}

\rightline{SISSA Preprint 40/2017/MATE}

\bigskip\bigskip\bigskip

 \begin{center}
{\Large\bf  SEMISTABLE HIGGS BUNDLES ON CALABI-YAU MANIFOLDS} \\[30pt]

{\sc U. Bruzzo,$^{\S\sharp\flat}$ V. Lanza$^\ddag$ and A. Lo Giudice$^\P$}\\[15pt]
$^\S$ SISSA (Scuola Internazionale Superiore di Studi Avanzati), \\ Via Bonomea 265, 34136 Trieste, Italy  \\[3pt]
$^\sharp$  Istituto Nazionale di Fisica Nucleare, Sezione di Trieste\\[3pt]
$^\flat$ Arnold-Regge Center,  via P.~Giuria 1,  10125 Torino, Italy \\[3pt]
 $\ddag$ IMECC - UNICAMP, Departamento de Matem\'atica, \\ 
 Rua S\'ergio Buarque de Holanda, 651, 13083-970 Campinas-SP, Brazil   \\[3pt] 
$^\P$ Ripa di Porta Ticinese 93, 20021 Milano, Italy
 \end{center} 

 \bigskip  \bigskip {\small
\begin{quote}{\sc Abstract.}  We provide a partial classification of semistable Higgs bundles over a simply connected Calabi-Yau manifolds. Applications to a conjecture about a special class of semistable Higgs bundles are given. In particular, the conjecture is proved for K3 and Enriques surfaces, and some related classes of surfaces.
\end{quote}
}

\bigskip
 
 \let\svthefootnote\thefootnote
\let\thefootnote\relax\footnote{
\hskip-\parindent {\em Date: } \today  
Email: {\tt  bruzzo@sissa.it, valeriano@ime.unicamp.br, alessiologiudic@gmail.com}
}
\addtocounter{footnote}{-1}\let\thefootnote\svthefootnote

 \bigskip
 \section{Introduction}
 
 In \cite{BBGL} it was shown that the Higgs field of a polystable Higgs bundle over a simply connected Calabi-Yau manifold $X$ necessarily vanishes. This result does not hold true in the semistable case, as the following example shows: take
$E = \Ol_X\oplus \Omega_X^1$ with Higgs field $\phi(f,\omega)=(\omega,0)$. As the cotangent bundle $\Omega_X^1$
is polystable, $E$ is polystable as well,  hence semistable;  so $(E,\phi)$ is a fortiori Higgs-semistable, but the Higgs field is nonzero
(note that  $(E,\phi)$ is not polystable as a Higgs bundle).

Nevertheless, possible Higgs fields for semistable Higgs bundles on Calabi-Yau manifolds are strongly constrained. For rank 2, no nontrivial Higgs fields exist. The same is true for rank 3 when $\dim X \ge 3$. This will be shown in section \ref{lowrank}. Section \ref{sec:K3} treats the case of rank 3 semistable Higgs bundles on K3 surfaces, offering a classification which depends on the rank of the kernel of the Higgs field.
In section \ref{vanishingtf} we extend the result of \cite{BBGL} to the torsion-free, non-locally-free case; this will be needed later on, but it is also a result of some interest in itself.

The existence of semistable Higgs bundles is related to a conjecture which was established in \cite{BG3} 
(see \cite{BRcong} for a review). The conjecture is about an extension of the following result, first proved in \cite{NakayamaNo},
which generalizes to higher dimensions Miyaoka's semistability criterion for bundles on curves \cite{Mi}.
Let  $E$ be a vector bundle on a smooth polarized variety $(X,H)$, and let $\Delta(E)\in H^4(X,\Q)$ be its
discriminant, i.e., the characteristic class
$$\Delta(E) = c_2(E) -\frac{r-1}{2r}c_1(E)^2,$$
where $r = \rk E$. 

\begin{thm}\label{TeoBO}  The following conditions are equivalent:
\begin{enumerate}
   \item $E$ is semistable, and $\Delta(E)\cdot H^{n-2}=0$;
  \item for any morphism $f\colon C\to X$, where $C$ is a smooth projective curve, the vector bundle $f^*(E)$ is semistable.
\end{enumerate}
\end{thm}

The property in (ii), both for ordinary   and Higgs bundles, will be called {\em curve semistability.}

Actually using Theorem 2 in \cite{S92} the condition $\Delta(E)\cdot H^{n-2}=0$ is shown to be equivalent --- for a semistable bundle --- to $\Delta(E)=0$. In this form Theorem \ref{TeoBO} was stated and proved with different techniques in \cite{BHR}. It is now quite natural to ask if this theorem holds true also for Higgs bundles. As it was proved in \cite{BHR}, one has:

\begin{thm}  A semistable Higgs bundle   $(E,\phi)$ on $(X,H)$ with $\Delta(E) = 0$ is curve  semistable.
\end{thm}

The conjecture is that also the converse implication holds. This was proved for varieties with nef tangent bundle in \cite{BruLoGiu}, and for some other classes of varieties obtained from these by some easy geometric constructions. In section \ref{ConjK3} we offer a proof of the conjecture for K3 surfaces; then, using results from \cite{BruLoGiu}, the conjecture also holds for Enriques surfaces
(using the language of \cite{BruLoGiu}, we thus prove that K3 and Enriques surfaces are {\em Higgs varieties}.) The proof will use the results of the previous sections, and the theory of {\em Higgs numerically flat} Higgs bundles \cite{BG3,BBG}.

As a related result, in section \ref{appltoconj} we  show that on K3 surfaces there are no rank 3 curve semistable Higgs bundles with nonzero Higgs field. 
 
 In sections \ref{vanishingtf} and \ref{lowrank} we shall consider bundles on compact K\"ahler manifolds,
 while in sections \ref{sec:K3} through  \ref{ConjK3} we shall restrict to smooth projective varieties over $\mathbb C$.
 
 We conclude this introduction with the basic definitions about Higgs sheaves.
 Let $X$ be an $n$-dimensional compact complex manifold $X$ equipped with a K\"ahler form $\omega$.
The degree of a   coherent $\Ol_X$--module $F$ is defined as
$$\deg F  = \int_X c_1(F) \wedge \omega^{n-1}$$
and if $F$ has positive rank its slope is defined  as 
$$\mu( F) = \frac{\deg F}{\rk  F}. $$
If $X$ is a smooth algebraic variety over $\mathbb C$ with a polarization $H$, one can give an intersection-theoretic definition
of the degree
$$ \deg F = c_1(F) \cdot H^{n-1}.$$

\begin{defin} A Higgs sheaf  on $X$ is a pair $(E,\phi)$, where $E$ is a torsion-free
coherent sheaf on $X$  and $\phi  \colon E  \to  E \otimes \Omega_X^1$ is a morphism of
$\Ol_X$-modules
such that $\phi\wedge\phi= 0$, where $\phi\wedge\phi$ is the composition
$$ E \xrightarrow{\phi}   E \otimes  \Omega_X^1
\xrightarrow{\phi\otimes \operatorname{id}}  E \otimes  \Omega_X^1 \otimes  \Omega_X^1 \to  E\otimes  \Omega_X^2.$$
A Higgs subsheaf
 of a Higgs sheaf $(E,\phi)$ is a   subsheaf $G$ of $E$
such that $\phi(G) \subset G\otimes\Omega_X^1$.
A Higgs bundle is a Higgs sheaf  whose underlying coherent sheaf is locally free.

If  $(E,\phi)$ and $ (G,\psi)$ are
Higgs sheaves, a morphism $f\colon (E,\phi) \to (G,\psi) $ is a homomorphism of $\Ol_X$-modules
$f\colon  E\to G$ such that the diagram
$$\xymatrix{
E\ar[r]^f \ar[d]_\phi & G \ar[d]^\psi \\
E\otimes\Omega^1_X \ar[r]^{f\otimes\id} & G\otimes\Omega^1_X }
$$
commutes.
\label{defHnef}
\end{defin}

\begin{defin} A Higgs sheaf $ (E,\phi) $ is semistable (respectively, stable) if
$\mu(G)\le \mu(E)$ (respectively, $\mu(G)< \mu(E)$) for every
Higgs subsheaf $G$ of   $ (E,\phi) $ with $0< \rk  G < \rk  E$. It is polystable if it is the direct sum  of stable Higgs sheaf having the same slope.\end{defin}

 \smallskip\noindent {\bf Acknowledgements.} The research this paper is based upon was mostly made during the authors' visits to the University of Pennsylvania (Philadelphia) and Universidade Estadual de Campinas (UNICAMP). The authors are thankful to those institutions for the ospitality and   support. This research is partly supported by INdAM, and  PRIN ``Geometria delle variet\`a algebriche''. V.L. is supported by  the  FAPESP post-doctoral grant No.~2015/07766-4.  U.B. is a member of the VBAC group.

 \bigskip

 \section{Vanishing of Higgs fields for polystable sheaves}\label{vanishingtf}
 
 In this section we extend to the torsion-free case the following result {\cite{BBGL}.
Here by Calabi-Yau manifold we mean a K\"ahler manifold with vanishing first Chern class. We assume the manifold to be compact and connected. 
  
 \begin{thm} \cite[Corollary 2.6]{BBGL}\label{thm:Bisgiudice}
 For any polystable Higgs bundle over a simply connected, smooth Calabi-Yau manifold the Higgs field is identically zero.
\end{thm}

The generalization is easily done when the variety is a K3 surface $S$. In fact, let $(E,\phi)$ be a polystable Higgs sheaf over $S$. We have the commutative diagram
\[
\xymatrix{
  0\ar[r]&E\ar[r]\ar[d]^\phi&E^{\vee\vee}\ar[d]^{\phi^{\vee\vee}}\\
  0\ar[r]&E\otimes\Omega^1_S\ar[r]&E^{\vee\vee}\otimes\Omega_S^1
}
\]
where $\phi^\vee$ is defined as the composition
\[
\xymatrix{
  E^\vee\ar[r]&E^\vee\otimes T_S\otimes\Omega^1_S\ar[rr]^-{\phi^\ast\otimes\id_{\Omega_S^1}}&&E^\vee\otimes\Omega^1_S
  }
\]
and analogously $\phi^{\vee\vee}$. The Higgs bundle $(E^{\vee\vee},\phi^{\vee\vee})$ is polystable, so $\phi^{\vee\vee}=0$, and by the diagram also $\phi$ must vanish. 
 
For the higher-dimensional case, we start by recalling a couple of definitions:
\begin{defin}
 Let $(E,\phi)$ be a torsion-free sheaf on a  compact K\"ahler manifold $X$. We denote by $\omega$ the K\"ahler form of $X$, and by $S$ the locus where $E$ fails to be locally free (which has codimension at least $2$). An hermitian metric $h$ on the vector bundle $E|_{X\setminus S}$ is said to be \emph{admissible} if:
 \begin{enumerate}
  \item the curvature form $F_h$ of the associated Chern connection $\nabla^h$  is square integrable, and
  \item the mean curvature $\operatorname{tr}_\omega(F_h)$ is bounded.
 \end{enumerate}
 
 An admissible hermitian metric $h$ on $(E,\phi)$ is called \emph{hermitian Yang-Mills-Higgs} if the equation
 \[
  \operatorname{tr}_\omega\left(F_h+[\phi,\phi^\ast]\right)=\lambda\cdot\id_E
 \]
is satisfied on $X\setminus S$ for some $\lambda\in\Com$.
\end{defin}

The following result has been proved by Biswas and Schumacher:
\begin{thm}\cite[Corollary 3.5]{BSchYM}
 A Higgs sheaf $(E,\phi)$ on a compact K\"ahler manifold $X$ is polystable if and only if there exists an hermitian Yang-Mills-Higgs metric $h$ on it (uniquely determined up to homotheties). In particular, the Yang-Mills-Higgs connection $\nabla^h$ is unique.
\end{thm}

As a direct consequence of \cite[Proposition 3.3]{S88} one has:
\begin{lemma}\label{lemmaSimp}
 Let $(E,\phi)$ be a polystable Higgs bundle over a connected open dense subset $Y$ of a compact    K\"ahler manifold, and let $F\subset E$ be a saturated Higgs subsheaf with $\mu(F)=\mu(E)$. Then:
 \begin{itemize}
  \item $F$ is a subbundle;
  \item $F^\perp$, where the orthogonal is taken with respect to the hermitian Yang-Mills-Higgs metric $h$ on $Y$, is a 
  holomorphic bundle;
  \item $E\simeq F\oplus F^\perp$.
 \end{itemize}
\end{lemma}

The key observation to prove Theorem \ref{thm:Bisgiudice} is given by \cite[Proposition 2.2]{BBGL}, which we briefly recall here. Let $(E,\phi)$ be a polystable Higgs bundle on a K\"ahler manifold $X$ and denote by $h$ the hermitian Yang-Mills-Higgs metric on $E$. If we denote by $\nabla^{\omega,h}$ the connection on $\End(E)\otimes\Omega^1_X$ induced by $\nabla^h$ and by the Levi-Civita connection on $\Omega^1_X$ associated with $\omega$, then one proves $\nabla^{\omega,h}\phi=0$ (i.e., $\phi$ is flat as a section of $\End(E)\otimes\Omega_X^1$). This needs Lemma \ref{lemmaSimp}. Moreover, let $Z$ be the regular locus of $E$;
note that $\operatorname{codim} (X\setminus Z) \ge 2$.
One uses the pointwise eigenspace decomposition of $\phi$, and the fact that $X$ (and therefore also $Z$) is simply connected to establish canonical isomorphisms
$$T_x X\simeq T_yX,\qquad E_x\simeq E_y$$
for every pair of points $x$, $y\in Z$. Then the pointwise eigenspace decomposition of $E$ produces a direct sum splitting
$$ E\vert_{Z} = \bigoplus_{i=1}^m F_i.$$
There is an induced Higgs field $\phi_i$ on $F_i$, and one defines
$$\tilde \phi_i = \phi_i - \frac{1}{\rk F_i} \operatorname{tr} \phi_i \otimes \operatorname{id}_{F_i}.$$
Proceeding as in  \cite[Proposition 2.5]{BBGL} one shows that $\tilde \phi_i =0$, so that 
$$ \phi -  \frac{1}{\rk E} \operatorname{tr} \phi \otimes \operatorname{id}_{E} =0$$
on $Z$, and therefore on $X$, and since $H^0(X,\Omega^1_X)=0$, we eventually have $\phi = 0 $.

In this way we have proved the needed generalization of Theorem \ref{thm:Bisgiudice}, as expressed
by the following result.

 \begin{thm} \label{thm:BisgiudiceTF}
 For any polystable Higgs sheaf over a simply connected Calabi-Yau manifold the Higgs field is identically zero.
\end{thm}

\begin{cor}\label{cor:gr}
 Let $(E,\phi)$ be a semistable Higgs sheaf over $X$. Then $\gr(\phi)=0$ and $\phi$ is nilpotent.
\end{cor}
\begin{proof}
$\gr(\phi)$ vanishes  as  $(\gr(E),\gr(\phi))$ is a polystable Higgs sheaf.

We claim that $\phi^s=0$, where $X$ is the length of the Jordan-H\"older filtration $\{(E_i,\phi_i)\}$. Indeed this is equivalent to show that $\phi(E_i)\subseteq E_{i-1}\otimes\om{X}$ for any $i\in\{1,\dots,s\}$, and this is ensured by the condition $\gr(\phi)=0$.
\end{proof}
Note that the  implication $\gr(\phi)=0\Rightarrow \phi$ is nilpotent holds for any variety. 

\bigskip
\section{Vanishing of Higgs fields for semistable sheaves of low rank}\label{lowrank} 
 
 We begin by presenting some preliminary results. $X$ will be a compact K\"ahler manifold. 
Given a holomorphic vector bundle $M$ on   $X$, an  \emph{$M$-pair} is a pair $(E,\phi)$, where $E$ is a torsion-free sheaf on $X$, and $\phi\colon E\to E\otimes M$ is a morphism.
 An $M$-pair  is said to be  semistable  if, for any $\phi$-invariant subsheaf $F$ of $E$, $\mu(F)\leq\mu(E)$.

\begin{prop}\label{prop:holp}
 Let $(E,\phi)$ be a semistable $M$-pair. If $M$ is semistable and has nonpositive degree, then $E$ is semistable as a sheaf.
\end{prop}
 \begin{proof}
  The proof of \cite[Theorem 2.10]{BruLoGiu} holds true also when $E$ is only torsion-free.
 \end{proof}

\begin{lemma}\label{lemma:kerss} \cite[Proposition 5.5.(2)]{BRcong}
Suppose that $\Omega_X^1$ is semistable and has degree zero.  Then, for any semistable Higgs sheaf $(E,\phi)$, the kernel and the image of $\phi$ are semistable Higgs sheaves with the same slope as $E$.
\end{lemma}

\begin{proof}
 Let $(E,\phi)$ be a semistable Higgs sheaf. By Proposition \ref{prop:holp}, $E$ is semistable as a sheaf; in particular, also $E\otimes \Omega_X^1$ is semistable, with $\mu(E\otimes \Omega_X^1)=\mu(E)$. As $K:=\ker\phi$ is a subsheaf of $E$, and $Q:=\im\phi$ is a subsheaf of $E\otimes \Omega_X^1$ (note that, in particular, this implies that $Q$ is torsion-free), we have \[\mu(K)\leq \mu(E)\qquad \text{and} \qquad\mu(Q)\leq \mu(E\otimes\Omega_X^1)=\mu(E)\,.\] Finally, by considering the exact sequence
\begin{equation}\label{lasolita}
0\to K \to E \to Q\to 0\,,
\end{equation}
one gets that the only possibility is $\mu(K)=\mu(E)=\mu(Q)$. Being a subsheaf and, respectively, a Higgs quotient with the same slope as $E$, $K$ and $Q$ are both semistable.
\end{proof}

For later purposes, let us suppose that the induced Higgs field on $Q$ vanishes. Under this assumption, one obtains the diagram:
\begin{equation}\label{eq:diag1}
 \xymatrix{&&0\ar[d]\\
         & K\ar@{=}[d]\ar@{=}[r]             & K\ar[d]\ar^-0[r]               & Q\ar@{=}[d]                        &   \\
  0\ar[r]& K\ar[d]^-0\ar[r]                  & E\ar[r]\ar[d]^-\phi            & Q\ar[r]\ar[d]^-0                   & 0 \\
  0\ar[r]& K\otimes\Omega_X^1\ar@{=}[d]\ar[r]& E\otimes\Omega_X^1\ar[r]\ar[d] & Q\otimes\Omega_X^1\ar[r]\ar@{=}[d] & 0 \\
         & K\otimes\Omega_X^1\ar[r]          & C\ar[d]\ar[r]                  & Q\otimes\Omega_X^1\ar[r]           & 0 \\
         &                                   & 0
  }
\end{equation}
where $C$ is defined as the cokernel of $\phi$. By the snake lemma we get the exact sequence
 \begin{equation}\label{eq:seqC}
  \xymatrix{
   0\ar[r]&Q\ar[r]^-\iota&K\otimes\Omega_X^1\ar[r]&C\ar[r]&Q\otimes\Omega^1_X\ar[r]&0\,;
  }
 \end{equation}
notice also that $\mu(K\otimes\Omega_X^1)=\mu(Q\otimes\Omega_X^1)=\mu(E)$. Diagram \eqref{eq:diag1} and eq.~\eqref{eq:seqC} will be repeatedly used throughout the paper.

 \begin{lemma}\label{lemma:rk1tf}
 Consider a rank $1$ torsion-free sheaf $L$, and a locally free sheaf $E$ on   $X$. If $E$ is stable, $L\otimes E$ is stable as well.
\end{lemma}

\begin{proof}
 Let $F$ be a destabilizing subsheaf for $L\otimes E$; i.e., $F\subset L\otimes E$ and $\mu(F)\geq\mu(L\otimes E)=\mu(L)+\mu(E)$. Consider the composition
 \[
  F\otimes L^\vee\to E\otimes L\otimes L^\vee\to E\,;
 \]
 the support of its kernel $T$ is a closed subset of $X$ of codimension at least $2$, since the map is injective where $L$ is locally free. The quotient $G=F\otimes L^\vee/T$ destabilizes $E$:  by Riemann-Roch Theorem, $c_1(T)=0$, and this implies that $\mu(G)=\mu(F)+\mu(L^\vee)\geq\mu(L)+\mu(E)+\mu(L^\vee)=\mu(E)$.
\end{proof}

We fix now a simply connected  Calabi-Yau manifold $X$ of dimension $n\geq2$; from now on, however, by Calabi-Yau $n$-manifold we shall mean a compact complex manifold with (maximal) $SU(n)$ holonomy. Note that, as a consequence, the  tangent bundle $T_X$ is stable. Note also that Proposition \ref{prop:holp} and Lemma \ref{lemma:kerss} hold true on $X$.

As rank $1$ sheaves are  stable, by Theorem \ref{thm:BisgiudiceTF} there is no nonzero Higgs fields on a rank $1$ Higgs sheaf on $X$. More trivially, since the double dual of a rank $1$ torsion-free sheaf is always locally free (see e.g.~Lemma 1  in \cite{Barth-stable}; the proof works also in the complex analytic case), if $\phi\not= 0$, $\phi^{\vee\vee}$ would provide a nonzero element of $H^0(\Omega^1_X)=0$.

\begin{thm}\label{thm:main2} One has the following vanishing results:
 \begin{enumerate}
  \item Any rank $2$ semistable Higgs sheaf $(E,\phi)$ on $X$ has $\phi= 0$.
  \item If $\dim(X)\geq 3$, any rank $3$ semistable Higgs sheaf $(E,\phi)$ on $X$ has $\phi=0$.
 \end{enumerate}
\end{thm}
\begin{proof}

In both cases, if $(E,\phi)$ is stable, we can apply Theorem \ref{thm:BisgiudiceTF} to conclude.
Indeed, suppose that $(E,\phi)$ is properly semistable and $\phi\not= 0$. Let $K$ and $Q$ be as in \eqref{lasolita}, and assume $\rk(K)=1$.


Suppose that the Higgs field $\bar{\phi}$ induced by $\phi$ on $Q$ vanishes, so that one can consider the sequence in \eqref{eq:seqC}. By Lemma \ref{lemma:rk1tf}, $K\otimes\Omega^1_X$ is a stable rank $n$ sheaf. Therefore, by \cite[Proposition 5.7.11]{Ko}, the map $\iota$ in \eqref{eq:seqC} is either generically surjective or zero. In the first case, we get a contradiction since $\rk(Q)<\rk(K\otimes\Omega^1_X)=n$ (in both settings (i) and (ii)); but also in the other case we get a contradiction, as $Q$ is not the zero sheaf.

If $\rk(E)=2$, then $\rk(Q)=1$ and $\bar{\phi}=0$ by the previous considerations. This proves (i).

If $\rk(E)=3$, then $\rk(Q)=2$ and the vanishing of $\bar{\phi}$ is given by (i).

It remains to deal with the case in which $E$ has rank $3$ and $K$ has rank $2$. Since $\rk(Q)=1$, $\bar{\phi}=0$. We dualize the sequence \eqref{eq:seqC}; setting $G = E^\vee/Q^\vee$, we obtain the diagram
\begin{equation}\label{eq:diag3}
 \xymatrix{ && 0 \ar[d] \\
         & Q^\vee \ar@{=}[d]\ar@{=}[r]             & Q^\vee \ar[d]\ar^-0[r]               & G\ar@{=}[d]                        &   \\
  0\ar[r]& Q^\vee \ar[d]^-0\ar[r]                  & E^\vee \ar[r]\ar[d]^-{\phi^\vee}            & G\ar[r]\ar[d]^-0                   & 0 \\
  0\ar[r]& Q^\vee \otimes\Omega_X^1\ar@{=}[d]\ar[r]& E^\vee \otimes\Omega_X^1\ar[r]\ar[d] & G \otimes\Omega_X^1\ar[r]\ar@{=}[d] & 0 \\
         & Q^\vee \otimes\Omega_X^1\ar[r]          & D \ar[d]\ar[r]                  & G\otimes\Omega_X^1\ar[r]           & 0 \\
         &                                   & 0 }
\end{equation}
with $\rk(Q^\vee)=1$, and we can conclude reasoning as in the first part of the proof.
\end{proof}

\bigskip
\section{Rank $3$ semistable Higgs bundles with nonzero Higgs fields on K3 surfaces }\label{sec:K3}
For rank 3 semistable Higgs bundles on K3 surfaces an analogue of Theorem \ref{thm:main2} is false, as it is shown in the next example. Throughout this section, $S$ will be an algebraic K3 surface.

\begin{ex}\label{toyexm} Let $E = \Ol_S\oplus \Omega_S^1$ with Higgs field $\phi(f,\omega)=(\omega,0)$ \cite{S88}. This is semistable as an ordinary bundle, hence Higgs-semistable a fortiori, but the Higgs field is nonzero. We have $K = \Ol_S$ and $Q= \Omega_S^1$. Note that $E$ is polystable as a  bundle but not as a Higgs bundle.
\end{ex}

We shall give a classification theorem for rank 3 semistable Higgs bundles on a K3 surface having nonzero Higgs field, distinguishing the case where $\rk K=1$ or $2$ (note that $\rk K \ge 1$ by Corollary \ref{cor:gr}). We shall need the following result \cite[Lemma II.1.16]{OSS}.

\begin{lemma}  \label{locfree} 
On a smooth projective variety, any saturated subsheaf of a reflexive sheaf is reflexive. 
\end{lemma}


%
%
%

For rank $3$ Higgs sheaves on K3 surfaces Lemma \ref{lemma:kerss} can be improved.

\begin{lemma}\label{lemma:KQstable}
 Let $(E,\phi)$ be a rank $3$ semistable Higgs sheaf on $S$. Both $K=\ker\phi$ and $Q=\im\phi$ are stable (Higgs) sheaves.
\end{lemma}
\begin{proof}
 When $K$ is of rank $1$, the stability of $Q$ is given by the inclusion $\iota$ in \eqref{eq:seqC} (since $K\otimes \Omega^1_S$ is stable and the two sheaves have the same slope).

When $K$ has rank $2$,  $Q$ is stable as it has  rank $1$. On the other hand   $K$ is stable as well:  the injection
\begin{equation}\label{eq:Ginj}
 0\to G\to Q^\vee\otimes\Omega_S^1
\end{equation}
coming from diagram \eqref{eq:diag3}
implies that $G$ is stable. Now, dualizing \eqref{lasolita} twice, we have an isomorphism 
$G^\vee \simeq K^{\vee\vee}$, so that $K$ is stable.
\end{proof}

Lemma \ref{lemma:KQstable} shows in particular that for rank $3$ semistable Higgs bundles on K3 surfaces the Jordan-H\"older filtration never has maximal length, as in both cases $\rk K = 1, \, 2$,   a Jordan-H\"older filtration for $(E,\phi)$ is provided by $0\subset K\subset E$.

\subsection{Classification in the $\rk K=1$ case} We take a rank 3 Higgs bundle $(E,\phi)$ on the K3 surface $S$, 
and assume that $K = \ker \phi$ has rank 1, so that, by Lemma \ref{locfree}, it is a line bundle. We have the exact sequence of morphisms of Higgs bundles \eqref{lasolita}.
As $Q$ is stable (Lemma \ref{lemma:KQstable}), one has $Q^\vee\simeq K^\vee\otimes T_S$.  Moreover, by dualizing the sequence \eqref{lasolita} we get
\begin{equation}\label{eq:G} 0 \to Q^\vee \to E^\vee \to G \to 0 \end{equation}
where $G$ has the form $G = \mathcal I_Z(-K)$ for a 0-cycle $Z$ of length
$$ \ell =  - \operatorname{ch}_2(E) -24 + \tfrac32\kappa^2$$
with $\kappa = c_1(K)$.  Dualizing again we obtain
\begin{equation} \label{quot} 0 \to K \to E \to Q^{\vee\vee}  \xrightarrow{q} \mathcal Ext^1(\mathcal I_Z(-K),\Ol_S) \to 0 ,
\end{equation}
and one has
$$ \mathcal Ext^1(\mathcal I_Z(-K),\Ol_S) \simeq \Ol_Z(K).$$

We extract the following data.

\begin{enumerate}
\item The class $\kappa$  in $\operatorname{Pic}(S)$ of the line bundle $K$.
\item The 0-cycle $Z$.
\item A surjective morphism $q\colon \Omega_S^1\to \Ol_Z$.
\item A locally free extension
\begin{equation}\label{extiso} 
0 \to \Ol_S \to F \xrightarrow{f} \ker q \to 0 \,.
\end{equation} 
\end{enumerate}

Note that if we replace $(E,\phi)$ by an isomorphic Higgs bundle  $(E',\phi')$
the data $\kappa$ and $Z$ remain unchanged, while $q$ changes by an automorphism
of $ \Omega_S^1$, and the extension \eqref{extiso} is replaced by a new one which is
isomorphic to the previous one as a complex of sheaves.

Conversely, if such collection of data (the class $\kappa$, the 0-cycle $Z$, a surjective morphism $q\colon \Omega_S^1\to \Ol_Z$ modulo automorphisms
of $ \Omega_S^1$, and an extension as in  \eqref{extiso})
is given, we set $E=F(\kappa)$ and $Q = \ker q \otimes \Ol(\kappa)$, and define a Higgs field $\phi$ for $E$ as the composition 
$$ E \xrightarrow{f\otimes\mbox{\tiny id}} Q \to \Omega_S^1(\kappa) \to E \otimes \Omega^1_S\,.$$

 \subsection{Classification in the $\rk K=2$ case} We consider again the exact sequence \eqref{lasolita} but this time we assume that $\rk K=2$ (an example is provided by the dual to the Higgs bundle in Example \ref{toyexm}). By Lemma \ref{locfree}, $K$ is locally free.   The
 sheaf $Q$, which is torsion-free of rank one, 
is of the form $\mathcal I_Z(\gamma)$, where $\gamma = c_1(Q)$, and
 $Z$ is a suitable 0-cycle. Eventually, note that 
 \begin{equation}\label{identification} K\simeq T_S(\gamma).\end{equation}

 Again, we can extract the following data.
\begin{enumerate}
\item The class $\gamma=c_1(Q)$ in $\operatorname{Pic}(S)$.
\item The 0-cycle $Z$.
\item A locally free extension 
\begin{equation}\label{extiso2} 
0 \to T_S \to F \to \mathcal I_Z  \to 0 \,.
\end{equation}
\item In view of the isomorphism \eqref{identification} the Higgs field of $E$ yields a morphism
$j\colon \mathcal I_Z \to T_S\otimes\Omega^1_S$. \end{enumerate}
If we replace $(E,\phi)$ by an isomorphic Higgs bundle  $(E',\phi')$
the data $\gamma$ and $Z$ remain unchanged, while $j$ changes by an automorphism
of $ T_S$, and the extension \eqref{extiso2} is replaced by a new one which is
isomorphic to the previous one as a complex of sheaves. 
  
 {\it Vice versa,} if such data are given (the class $\gamma$, the 0-cycle $Z$, a morphism $\mathcal I_Z \to T_S\otimes\Omega^1_S$ modulo automorphisms
of $T_S$, and an extension as in  \eqref{extiso2}), it is immediate to define a vector bundle $E$ sitting in a sequence
 $$ 0 \to K \to E \to \mathcal I_Z(\gamma) \to 0 $$
 with a Higgs field $\phi$  given by $j(\kappa)$, so that $K$ is the kernel of $\phi$.
 
 In both cases ($\rk K =1$ or 2), it is easy to check that starting from a Higgs bundle $(E,\phi)$, collecting the data as above, and reconstructing a Higgs bundle,
we obtain the  Higgs bundle $(E,\phi)$ back up to isomorphism.

\bigskip 
\section{Analysis of curve semistability}\label{appltoconj}

Curve semistability, as reminded in the Introduction, is the property for a (Higgs) bundle to be (Higgs) semistable whenever pulled back to smooth projective curves; i.e., a (Higgs) bundle $E$ on smooth projective variety $X$ is said to be curve semistable  if $f^\ast(E)$ is semistable for every $f\colon C\to X$, where $C$ is a smooth projective curve. In this section we prove the nonexistence of rank $3$ curve semistable Higgs bundles with nontrivial Higgs field on a K3 surface with Picard number 1; in addition, after that, we discuss to some extent the case of rank 4 for Calabi-Yau varieties. 
This  in particular proves  the conjecture described in the Introduction in the case of rank 3 Higgs bundles on K3 surfaces with Picard number 1; however, in the next section the conjecture  will be proved for Higgs bundles of any rank on any K3 surface.

Before proving the main result of the section, let us explain what is the underlying idea by means of Example \ref{toyexm}. The key point is that $Q=\Omega_S^1$ is a Higgs quotient of $E$ with vanishing Higgs field: in other words, it can be viewed as an ordinary bundle. Since $\Delta(\Omega_S^1)=24>0$, by   Theorem \ref{TeoBO} we find a morphism $f\colon C\to S$ such that $f^\ast Q$ is not semistable. We denote by $\bar{Q}$ a destabilizing quotient, and we claim that $\bar{Q}$ destabilizes also $f^\ast E$. But this is obvious, since $\mu(f^\ast\Omega_S^1)=\mu(f^\ast E)=0$, and a Higgs quotient of a Higgs quotient is a Higgs quotient as well.

In general we  basically  need to take care of two issues: the fact that $Q$ may be non-locally free, and the twisting. To overcome these problems, in what follows we put the additional hypothesis $\Pic(S)\simeq\Z$ (which is satisfied by the generic algebraic K3 surface).

\begin{prop}\label{prop:cssK3}
Let $(E,\phi)$ be a rank $3$ semistable Higgs bundle $(E,\phi)$ on an algebraic K3 surface $S$, with $\Pic (S) \simeq \Z$. If $\phi\not= 0$, there exists a morphism $f\colon C \to S$, where $C$ is a smooth projective curve, such that $f^{\ast}E$ is not semistable as a Higgs bundle. 
\end{prop}

\begin{proof}

Let $(E,\phi)$ be a rank $3$ semistable Higgs {bundle} on $S$. We begin by assuming $\rk K =1$. We know that $Q^{\vee\vee}\simeq K \otimes\Omega_S^1$. 
 By dualizing twice the sequence in \eqref{lasolita} we get
\begin{equation}\label{eq:conj1}
 E^{\vee\vee}\to K^{\vee\vee}\otimes\Omega_S^1 \to R \to 0\,,
\end{equation}
where $R$ is of rank $0$. Since $K^{\vee\vee}\otimes\Omega_S^1$ has vanishing Higgs field and  $\Delta(K^{\vee\vee}\otimes\Omega_S^1)=24>0$, by   Theorem \ref{TeoBO} we find a morphism $f\colon C\to S$ such that $f^\ast (K^{\vee\vee}\otimes\Omega_S^1)$ is not semistable; in other words, there exists a subsheaf $M$ of $f^\ast(K^{\vee\vee}\otimes\Omega_S^1)$ such that
\[
 \deg M>\frac{1}{2}\deg f^\ast(K^{\vee\vee}\otimes\Omega_S^1)=c_1(K)\cdot C
\]
(notice that $M$ is clearly also a Higgs subsheaf).

We claim that $f^\ast E$ is not Higgs semistable. Indeed, by pulling back the sequence in \eqref{eq:conj1}, and by recalling that $f^\ast E^{\vee\vee}\simeq (f^\ast E)^{\vee\vee}\simeq f^\ast E$, we get the following diagram:\[
\xymatrix{
&0\ar[d]&0\ar[d]& 0\ar[d] &\\
0\ar[r] &M_2\ar[d]\ar[r]&M\ar[d]\ar[r]&M_1\ar[d]\ar[r] & 0 \\
0\ar[r]&\operatorname{ker}\alpha\ar[r]\ar[d]&f^\ast(K^{\vee\vee}\otimes\Omega_S^1)\ar[r]^-{\alpha}\ar[d]&f^\ast R\ar[r]\ar[d] &0\\
0\ar[r]&Q_2\ar[r]\ar[d]&Q_1\ar[r]\ar[d]&C_1\ar[d]\ar[r]&0\\
&0&0&0&
}
\]
Note that $Q_2$ is a Higgs quotient of $\operatorname{ker}\alpha$, and therefore also a Higgs quotient of $f^\ast E$. Now, since $\rk C_1 = 0$,
\[
\deg Q_2=\deg Q_1-\deg C_1 \le \deg Q_1< c_1(K)\cdot C\,.
\]
The hypothesis on the Picard group $\Pic(S)\simeq \Z$, together with $\mu(E)=\mu(K)$, implies 
$ c_1(K)\cdot C = \mu(f^\ast E)$, so that  $\deg Q_2< \mu(f^\ast E)$, i.e., $Q_2$ destabilizes  $f^\ast E$, as wanted.

We assume now that $\rk K =2$. By dualizing    \eqref{eq:G} we get
\[
 E^\vee\to Q^\vee\otimes\Omega^1_S \to R\to 0\,,
\]
where $R$ is a rank $0$ sheaf and $\mu(E^\vee)=\mu(Q^\vee\otimes\Omega^1_S)$. Reasoning as in the previous case (just noticing that now all the slopes involved change sign) one can find a morphism $f\colon C\to S$ such that $f^\ast(E^\vee)$ is not semistable,
 and so also $f^\ast E$ is not semistable. 
\end{proof}

In what follows we fix a simply connected, smooth Calabi-Yau variety $X$ of dimension $n \ge 2$; $(E,\phi)$ will be a  rank 4 Higgs bundle on $X$ with nonvanishing Higgs field. We shall need the following fact.
\begin{lemma}\label{c2}
If $X$ is a   Calabi-Yau manifold of dimension $n\ge 3$, then $c_2(X) \ne 0$.
\end{lemma} 
\begin{proof} In a Calabi-Yau manifold $X$, the square of the curvature of the Chern connection of the Calabi-Yau metric is proportional to $c_2(X)$ (see eqs.~(IV.4.2,3) in \cite{Ko});  hence, the
condition $c_2(X)=0$ is equivalent to the vanishing of the curvature. But the vanishing of the curvature contradicts the hypothesis that $X$ has maximal holonomy. \end{proof}

\begin{prop} Let $K$ be the kernel of $\phi$ (so that $1\leq \rk(K)\leq 3$).
 \begin{enumerate} 
  \item If $\rk(K)=3$, then $n=3$; in this case, under the further assumption $\Pic(X)\simeq \Z$, $E$ is not curve semistable.
  \item If $\rk(K)=1$, then $n=2$ or $n=3$; in this case, for $n=3$, under the further assumption $\Pic(X)\simeq \Z$, $E$ is not curve semistable.
 \end{enumerate}
\end{prop}

 \begin{proof}
   (i) Consider the usual sequence in \eqref{lasolita}: $K$ has vanishing Higgs field by definition, while $Q$ has vanishing Higgs field since it is of rank $1$. Therefore, we get a diagram analogous to \eqref{eq:diag1}, and, by dualizing it, also the one in \eqref{eq:diag3}. The injection $0\to G\to Q^\vee\otimes\Omega_X^1$, since the last term is stable, has to be generically surjective; in particular, we must have $3=\rk(G)=\rk(Q^\vee\otimes\Omega_X^1)=n$.
   
   Now, let assume $\Pic(X) \simeq \Z$. Combining the morphism
   $E^\vee \to G $ as in eq.~\eqref{eq:G} with the morphism $G \to Q^\vee \otimes \Omega^1_X$ contained in a diagram
   analogous to \eqref{eq:diag3}, we obtain a sequence
    $$ E^\vee  \to  Q^\vee \otimes \Omega^1_X \to R \to 0 $$ where $R$ has rank 0 (note that $Q^\vee$ is locally free).     
Now, $\Delta( Q^\vee\otimes \Omega^1_X)=c_2(X)\ne 0$ again
 ensures the existence of a morphism $f\colon C\to X$, where $C$ is a smooth projective curve, such that $f^\ast ( Q^\vee\otimes \Omega^1_X)$ is not semistable. Proceeding as in Proposition \ref{prop:cssK3} one shows that $f^\ast E$ is not semistable.

(ii) We always refer to eq.~\eqref{lasolita}. If we assume $n\geq 3$, we can apply Theorem \ref{thm:main2} to conclude that the Higgs field induced on $Q$ vanishes. We get then the usual diagram in \eqref{eq:diag1}. The injection $0\to Q\to K\otimes\Omega_X^1$, since the last term is stable, has to be generically surjective; in particular, we must have $3=\rk(Q)=\rk(K\otimes\Omega_X^1)=n$.
   
 For $n=3$, assume $\Pic(X) \simeq \Z$; composing the morphism $E\to Q^{\vee\vee}$
 obtained dualizing twice the sequence \eqref{lasolita} with the double dual of the morphism $Q \to K\otimes \Omega^1_X$
 coming  from \eqref{eq:seqC} we obtain the sequence
 $$ E \to  K\otimes \Omega^1_X \to R \to 0 $$ where $R$ has rank 0.    Now one proceeds as in the $\rk K =3$ case.
 \end{proof}

 \begin{rem}
  When $\rk(K)=2$ the main obstruction in getting a classification lies in the fact that in general neither $K\otimes\Omega_X^1$ nor $Q^\vee\otimes\Omega^1_X$ are stable ($K$ and $Q^\vee$ have both rank $2$, so that we cannot apply Lemma \ref{lemma:rk1tf}).
 \end{rem}

 \bigskip
 \section{The conjecture for K3 surfaces}\label{ConjK3}
 As we discussed in the Introduction, for Higgs  bundles a full analogue of Theorem \ref{TeoBO}
 is still conjectural in general. So, we have the following
\begin{conjecture}\label{con}
Any curve semistable Higgs bundle on  a polarized variety $(X,H)$  has vanishing discriminant.
\end{conjecture}

This conjecture has been proved for varieties with nef tangent bundle in \cite{BruLoGiu}, and some other classes of varieties
obtained from these with some easy geometric constructions.
The set of varieties with nef tangent bundle contains all rational varieties, ruled surfaces and abelian varieties.

The results of the previous sections allow one to trivially prove the conjecture for rank 2 bundles on Calabi-Yau varieties and rank 3 bundles on Calabi-Yau varieties of dimension at least 3 (note that it makes sense to prove the conjecture for bundles of fixed rank). Indeed, in these cases the Higgs field necessarily vanishes (Theorem \ref{thm:main2}), and one is reduced to Theorem \ref{TeoBO}. The conjecture is easily showed to hold true also for rank 3 bundles on K3 surfaces with Picard number 1: it is sufficient to notice that Proposition \ref{prop:cssK3} implies that such a bundle or has vanishing Higgs field, and one is reduced again to Theorem \ref{TeoBO}, or is not curve semistable, and there is nothing to prove.

 In this section we prove Conjecture 1 for K3 surfaces. To this end we shall use some of the results of the previous sections, and  some properties of the so-called H-numerically flat (H-nflat) Higgs bundles. We recall here their basic definitions and   properties we shall need later on. See \cite{BG3,BBG} for full definitions and proofs.
 
 For the moment, let $X$ be any smooth projective variety. If  $E$ is a rank $r$ vector bundle on $X$, and   $s<r$  is a positive
integer, we    denote by $p_s\colon\operatorname{Gr}_s(E) \to X $ the Grassmann bundle parameterizing
rank $s$ locally free quotients of  $E$ of dimension $s$.  There is on $\operatorname{Gr}_s(E)$ an exact sequence of vector bundles 

\begin{equation}\label{univ}
0 \to S_{r-s,E} \,\stackrel{\psi}{\longrightarrow}\, p_s^* E \,
\stackrel{\eta}{\longrightarrow}\, Q_{s,E} \,.
\to 0
\end{equation}
Here  $S_{r-s,E}$ is
the universal  rank $r-s$ subbundle of $p_s^* E$ and 
 $Q_{s,E}$ is the universal   rank $s$ quotient.

Given a Higgs bundle $\Ea\,=\,(E,\phi) $, for every $s$ we define a closed subscheme
${\mathfrak G}r_s(\Ea)\subset\operatorname{Gr}_s(E)$ as the vanishing locus of the 
composite morphism
\begin{equation}\label{lambda}
(\eta\otimes\text{Id})\circ p_s^\ast(\phi) \circ \psi\,\colon\, S_{r-s,E}\to Q_{s,E}\otimes
 p_s^\ast\Omega_X^1\, .
\end{equation}
Let $\rho_s= p_s\vert_{{\mathfrak G}r_s(\Ea)}\colon {\mathfrak G}r_s(\Ea)\to X$ be the
restriction. The restriction of \eqref{univ} to ${\mathfrak G}r_s(\Ea)$ provides the universal
exact sequence
\begin{equation}\label{g1}
0\to \mathfrak S_{r-s,\Ea}\,\stackrel{\psi}{\longrightarrow} \,\rho_s^\ast \Ea\,
\stackrel{\eta}{\longrightarrow} \mathfrak Q_{s,\Ea}\to 0\, ,
\end{equation}
where  $\mathfrak Q_{s,\Ea}=Q_{s,E}\vert_{{\mathfrak G}r_s(\Ea)}$ 
is equipped with the quotient Higgs field induced by the Higgs field
$\rho_s^\ast \phi$. The scheme 
${\mathfrak G}r_s(\Ea)$ satisfies  the universal property   that  a morphism of varieties $f\colon T \rightarrow X$  factors through
${\mathfrak G}r_s(\Ea)$ if and only if the pullback $f^*(E)$ admits a locally free rank $s$ Higgs quotient. In that case the
pullback of the above universal sequence on ${\mathfrak G}r_s(E)$ gives the desired quotient of $f^*(E)$.

\begin{defin}\label{moddef}  
A Higgs bundle $\Ea=(E,\phi)$ of rank one is said to be Higgs-numerically
effective  (H-nef for short) if $E$ is numerically effective in the usual sense. If
$\rk \Ea \geq 2$, we inductively define H-nefness by requiring that
\begin{enumerate}
\item all Higgs bundles ${\mathfrak Q}_{s,\Ea}$ are H-nef   for all $s$, and

\item the determinant line bundle $\det(E)$ is nef.
\end{enumerate}
If both $\Ea$ and $\Ea^\ast$ are Higgs-numerically effective, $\Ea$ is said to
be Higgs-numerically flat (H-nflat).
\end{defin}

From Definition \ref{moddef} one sees that    the first Chern class of an
H-nflat Higgs bundle is numerically equivalent to zero. Note that if
$\Ea\,=\,(E,\,\phi)$, with $E$ nef in the usual sense, then $\Ea$ is H-nef. If
$\phi\,=\,0$, the Higgs bundle $\Ea\,=\,(E,\,0)$ is H-nef if
and only if $E$ is nef in the usual sense. 

We collect in the following Proposition the properties of  H-nef and H-nflat Higgs bundles we shall need later on.
\begin{prop}\label{results}\mbox{}
\begin{enumerate}
\item[(i)] An H-nflat Higgs bundle is semistable (with respect to any polarization).
\item[(ii)]  A curve semistable   Higgs bundle whose first Chern class is numerically equivalent  to zero is H-nflat. In particular,
\item[(iii)] A  semistable Higgs bundle on a curve having zero first Chern class in H-nflat.
\item[(iv)]  A Higgs bundle $\Ea$ on $X$ is H-nef (H-nflat) if and only if for every morphism $f\colon C\to X$, where
$C$ is a smooth projective curve, the Higgs bundle $f^\ast\Ea$ is H-nef (H-nflat). 
  \item[(v)]  The kernel and cokernel of  a morphism of H-nflat Higgs bundles are
H-nflat Higgs bundles.
 \end{enumerate}
\end{prop}

 
 We recall the following fact \cite[\S 3.2]{BRcong}.
 \begin{prop} Conjecture 1 is equivalent to the following statement: all Chern classes of an H-nflat Higgs bundle vanish.
 \end{prop}
 
 \begin{thm} Let $X$ be a K3 surface. If $(E,\phi)$ is an H-nflat Higgs bundle on $X$, then $c_i(E)=0$
 for $i>0$. So, Conjecture 1 is true for  K3 surfaces.
 \end{thm}

\begin{proof} We prove this result by induction on the rank of $E$. Since $(E,\phi)$ is H-nflat, it  is semistable (Prop.~\ref{results}(i))
and has $c_1(E)=0$. Moreover, it is curve semistable. We also know that $K=\ker\phi$ has rank at least one (Corollary \ref{cor:gr}) and again, it  is locally free by Lemma \ref{locfree}. We claim that $c_1(K)=0$. Indeed, given any polarization $H$ on $X$,
we have $c_1(K)\cdot H = 0 $  by Lemma \ref{lemma:kerss}, so that by the Hodge index theorem $c_1(K)^2<0$.
However, $c_1(K)\cdot H = 0 $ holds for {\em all}  polarizations $H$ in $X$, and since the ample cone is open in
$\Pic(X)\otimes\mathbb R$, we obtain that $c_1(K)$ lies in the trascendental lattice. But of course it also lies in the Picard lattice, and as
$c_1(K)^2 \ne 0$, we obtain $c_1(K)=0$.

Note that  if $K=E$ there is nothing to prove, as then $\phi=0$, $E$ is numerically flat, and its Chern classes vanish \cite{DPS}. So $ 1 \le \rk K < \rk E$.

We claim that $K$ is H-nflat as well. To show that let us consider a morphism
$f\colon C \to X$, where $C$ is a smooth projective curve, and pull the exact sequence \eqref{lasolita} back to $C$. Since $K$ is locally free, we get an exact sequence
$$ 0 \to f^\ast(K) \to f^\ast(E) \to f^\ast(Q) \to 0 .$$ 
Now $\mu(f^\ast(K)) =\mu(f^\ast(E))=0$, so that $f^\ast(K)$ is semistable and hence H-nflat (Prop.~\ref{results}(iii)); so by Prop.~\ref{results}(iv) $K$ is H-nflat as claimed.

By Prop.~\ref{results}(v), $Q$ is H-nflat as well. By induction, $K$ and $Q$ have vanishing Chern classes, so the same is true for $E$.

To start the induction,   note that for $\rk E=2$ the Higgs field $\phi$ vanishes (Theorem \ref{thm:main2}), so again
$E$ is numerically flat.
\end{proof}

Finally, Proposition 3.12 in  \cite{BruLoGiu} implies that the conjecture also holds for Enriques surfaces (and actually, more generally, for smooth projective surfaces $X$ such that there is  a finite \'etale map $Y\to X$ where $Y$ is a K3 surface); analogously, 
by  Proposition 3.13 in  \cite{BruLoGiu}, the conjecture holds for varieties $Y$ for which there is a surjective morphism onto
a K3 surface $Y$ whose fibres are smooth and   rationally connected (for instance, projective bundles on $Y$ and blowups of $Y$ at points).

\end{document}